\documentclass[a4paper,11pt,final]{article}
\usepackage{a4wide}

\usepackage[T1]{fontenc}
\usepackage{lmodern}

\usepackage[utf8]{inputenc}
\usepackage{amsmath}
\usepackage{amssymb}
\usepackage{amsthm}
\usepackage{graphicx}
\usepackage{showkeys}
\usepackage{listings}
\usepackage{enumerate}

\newcommand{\R}{\mathbb R}
\newcommand{\norm}[1]{\left\Vert #1 \right\Vert}
\DeclareMathOperator{\vettop}{vec}

\usepackage[ruled]{algorithm2e}

\usepackage{float}
\usepackage{hyperref}

\newtheorem{theorem}{Theorem}
\theoremstyle{remark}
\newtheorem{remark}[theorem]{Remark}

\newtheorem{lemma}[theorem]{Lemma}
\newtheorem{corollary}[theorem]{Corollary}

\newcommand{\vett}[2]{\begin{bmatrix}#1 \\ #2 \end{bmatrix}}
\newcommand{\twotwo}[4]{\begin{bmatrix}#1 & #2\\ #3 & #4\end{bmatrix}}

\author{Federico Poloni\thanks{Scuola Normale Superiore, piazza dei Cavalieri, 7; 56126 Pisa, Italy. Phone: +39 050 509111. Fax: +39 050 563513. E-mail \href{mailto:f.poloni@sns.it}{f.poloni@sns.it} }}
\title{Quadratic Vector Equations}
\date{}
\usepackage{microtype}
\begin{document}
\maketitle

\begin{abstract}
We study in a unified fashion several quadratic vector and matrix equations with nonnegativity hypotheses, by seeing them as special cases of the general problem
$
 Mx=a+b(x,x)
$,
where $a$ and the unknown $x$ are componentwise nonnegative vectors, $M$ is a nonsingular M-matrix, and $b$ is a bilinear map from pairs of nonnegative vectors to nonnegative vectors. Specific cases of this equation have been studied extensively in the past by several authors, and include unilateral matrix equations from queuing problems [Bini, Latouche, Meini, 2005], nonsymmetric algebraic Riccati equations [Guo, Laub, 2000], and quadratic matrix equations encountered in neutron transport theory [Lu, 2005]. 

We present a unified approach which treats the common aspects of their theoretical properties and basic iterative solution algorithms. This has interesting consequences: in some cases, we are able to derive in full generality theorems and proofs appeared in literature only for special cases of the problem; this broader view highlights the role of hypotheses such as the strict positivity of the minimal solution. In an example, we adapt an algorithm derived for one equation of the class to another, with computational advantage with respect to the existing methods. We discuss possible research lines, including the relationship among Newton-type methods and the cyclic reduction algorithm for unilateral quadratic equations.
\end{abstract}


\noindent {\bf MSC classes: } 15A24, 65F30

\section{Introduction}
In this paper, we aim to study in a unified fashion several quadratic vector and matrix equations with nonnegativity hypotheses. Specific cases of such problems have been studied extensively in the past by several authors. For references to the single equations and results, we refer the reader to the following sections, in particular \autoref{s:concrete}. Many of the results appearing here have already been proved for one or more of the single instances of the problems, resorting to specific characteristics of the problem. In some cases the proofs we present here are mere rewritings of the original proofs with a little change of notation to adapt them to our framework, but in some cases we extend the existing results to other problems of the considered class and understand the role of some key assumptions such as the positivity of the minimal solution.

It is worth noting that Ortega and Rheinboldt \cite[Chapter 13]{OrtegaRheinboldt}, in a 1970 book, treat a similar problem in a far more general setting, assuming only the monotonicity and operator convexity of the involved operator. Since their hypotheses are far more general than those of our problem, the obtained results are less precise than those we are reporting here. Moreover, all of their proofs have to be adapted to our case, since the operator $F(x)$ we are dealing with is operator concave instead of convex.

\paragraph{Useful results on $M$-matrices}
 In the following, $A \geq B$ (resp. $A>B$) means $A_{ij} \geq B_{ij}$ (resp. $A_{ij}>B_{ij}$) for all $i,j$. A real square matrix $Z$ is said \emph{$Z$-matrix} if $Z_{ij} \leq 0$ for all $i \neq j$. A $Z$-matrix is said an $M$-matrix if it can be written in the form $sI-P$, where $P\geq 0$ and $s\geq \rho(P)$ and $\rho(\cdot)$ denotes the spectral radius.

We make use on the following results.
\begin{theorem}\label{Mmatrices}
The following facts hold.
 \begin{enumerate}
  \item If $Z$ is a $Z$-matrix and there exists a vector $v> 0$ such that $Zv \geq 0$, then $Z$ is an M-matrix;\label{firstitem}
  \item If $Z$ is a $Z$-matrix and $Z \geq M$ for an $M$-matrix $M$, then $Z$ is an $M$-matrix. \label{Mmatrices2}
  \item A nonsingular $Z$-matrix $Z$ is an $M$-matrix if and only if $Z^{-1}\geq 0$.
  \item A $Z$-matrix $Z$ is a nonsingular $M$-matrix if and only if it has a representation as $N-P$, where $N$ is a nonsingular $M$-matrix, $P\geq 0$  and $\rho(N^{-1}P)<1$. \label{Mmatrices4}
  \item A $Z$-matrix $Z$ is an $M$-matrix if and only if it has a representation as $N-P$, where $N$ is a nonsingular $M$-matrix, $P\geq 0$  and $\rho(N^{-1}P)\leq 1$. \label{Mmatrices5}
 \end{enumerate}
\end{theorem}
\begin{proof}
 Items~\ref{firstitem}--\ref{Mmatrices4} are found in \cite{BermanPlemmons}. We report here a self-contained proof of item~\ref{Mmatrices5}, which was suggested by one of the referees of this paper.
\begin{itemize}
 \item[$\Rightarrow$] if $Z$ is an $M$-matrix, $Z=sI-B$ with $B\geq 0$ and $\rho(B)\leq s$. Then $N=(s+1)I$, $P=B+I$ is a splitting with the required properties.
 \item[$\Leftarrow$] Let $Z=N-P$ be a splitting with the stated properties, and let $N=tI-C$ with $C\geq 0$, $\rho(C)<t$. For any $\varepsilon>0$ we have $\rho(N^{-1}\frac{1}{1+\varepsilon}P)<1$, thus $Z_\varepsilon=N-\frac{1}{1+\varepsilon}P$ is a nonsingular $M$-matrix. Then $Z_\varepsilon=tI-(C+\frac{1}{1+\varepsilon}P)$, so $t>C+\frac{1}{1+\varepsilon}P$. Letting $\varepsilon\to 0^+$ gives $t\geq \rho(C+P)$. Therefore, $Z=tI-(C+P)$ is an $M$-matrix. \qedhere
\end{itemize}
\end{proof}
Moreover, we need the following extension of item~\ref{Mmatrices2}.
\begin{theorem}\label{twoconditions}
 Under the hypotheses of item~\ref{Mmatrices2} of Theorem~\ref{Mmatrices}, if at least one of the following two additional conditions holds
\begin{itemize}
 \item $M$ is nonsingular,
 \item $M$ is irreducible and $Z\neq M$,
\end{itemize}
then $Z$ is nonsingular.
\end{theorem}
\begin{proof}
The results follow from the fact that the Perron value of a nonnegative matrix is an nondecreasing function of its entries, and a strictly increasing one if the matrix is irreducible \cite{BermanPlemmons}.
\end{proof}

\section{General problem} 
We are interested in solving the equation
\begin{equation} \label{qve}
 Mx=a+b(x,x)
\end{equation}
(quadratic vector equation, QVE) where $M\in \R^{n \times n}$ is a nonsingular $M$-matrix, $a,x\in \R^n$, $a,x\geq 0$, and $b$ is a nonnegative vector bilinear form, i.e., a map $b: \R^n \times \R^n \to \R^n$ such that
\begin{enumerate}
 \item $b(v,\cdot)$ and $b(\cdot,v)$ are linear maps for each $v \in \R^n$ (bilinearity);
 \item $b(x,y) \geq 0$ for all $x,y \geq 0$ (nonnegativity).
\end{enumerate}
The map $b$ can be represented by a tensor $B_{ijk}$, in the sense that $b(x,y)_k=\sum_{i,j=1}^n B_{ijk} x_i y_j$. It is easy to prove that $0\leq x\leq y$ and $0\leq z \leq w$ imply $b(x,z)\leq b(y,w)$. If $N$ is a nonsingular M-matrix,  $N^{-1}B$ denotes the tensor representing the map $(x,y) \mapsto N^{-1}b(x,y)$. Note that, here and in the following, we do not require that $b$ be symmetric (that is, $b(x,y)=b(y,x)$ for all $x,y$): while in the equation only the quadratic form associated to $b$ is used, in the solution algorithms there are often terms of the form $b(x,y)$ with $x \neq y$. Since there are multiple ways to extend the quadratic form $b(x,x)$ to a bilinear map $b(x,y)$, this leaves more freedom in defining the actual solution algorithms.

We are only interested in nonnegative solutions $x_\ast \geq 0$; in the following, when referring to solutions of \eqref{qve} we shall always mean \emph{nonnegative} solutions only. A solution $x_\ast$ of \eqref{qve} is said \emph{minimal} if $x_\ast \leq y$ for any other solution $y$.

Later on, we give a necessary and sufficient condition for \eqref{qve} to have a minimal solution.

\section{Concrete cases}\label{s:concrete}
\begin{description}
 \item [E1: Markovian binary trees] in \cite{BeanKontoleonTaylor, HautphenneLatoucheRemiche}, the equation
\[
x=a+b(x,x),
\]
with the assumption that $e=(1,1,\dots,1)^T$ is a solution, arises from the study of Markovian binary trees.
\item [E2: Lu's simple equation] in \cite{Lu02, LuNewton}, the equation
\[
 \begin{cases}
  u= u \circ (Pv) +e,\\
  v= v \circ (\tilde Pu) +e,
 \end{cases}
\]
where $u,v \in \R^m$ are the unknowns, $e$ is as above, $P$ and $\tilde P$ are two given nonnegative $m \times m$ matrices, and $a \circ b$ denotes the Hadamard (component-wise) product, arises as a special form of a Riccati equation appearing in a neutron transport problem. By setting $w:=[u^T\,v^T]^T$, the equation takes the form \eqref{qve}. 
\item[E3: Nonsymmetric algebraic Riccati equation] in \cite{GuoNAREandWienerHopf}, the equation
\[
 XCX+B-AX-XD=0,
\]
where $X,B \in \R^{m_1 \times m_2}$, $C \in \R^{m_2 \times m_1}$, $A \in \R^{m_1 \times m_1}$, $D \in \R^{m_2 \times m_2}$, and
\begin{equation}\label{nareM}
 \mathcal M=\begin{bmatrix}
             D &-C\\ -B & A
            \end{bmatrix}
\end{equation}
is a nonsingular or singular irreducible M-matrix, is studied. Vectorizing everything, we get
\[
 (I \otimes A+D^T \otimes I) \vettop(X)=\vettop(B)+\vettop(XCX),
\]
which is in the form \eqref{qve} with $n=m_1m_2$.
\item[E4: Unilateral quadratic matrix equation] in several queuing problems \cite{BiniLatoucheMeiniBook}, the equation
\[
 X=A+BX+CX^2,
\]
with $A,B,C,X \in \R^{m \times m}$, $A,B,C\geq 0$, and $(A+B+C)e=e$, is considered. Vectorizing everything, we again get the same class of equations, with $n=m^2$: in fact, since $Be \leq e$, $(I-B)e \geq 0$ and thus $I-B$ is an M-matrix.
\end{description}

To ease the notation in the cases E3 and E4, in the following we shall set $x_k=\vettop(X_k)$, and for E3 also $m=\max(m_1,m_2)$.

\section{Minimal solution}
\paragraph{Existence of the minimal solution}
It is clear by considering the scalar case ($n=1$) that \eqref{qve} may have no real solutions. The following additional condition allows us to prove their existence. We call a linear map $l:\R^n\to \R^m$ \emph{weakly positive} if $l(x)\geq 0$, $l(x)\neq 0$ whenever $x\geq 0$, $x\neq 0$.
\begin{description}
 \item [Condition A1] There are a weakly positive map $l:\R^n \to \R^m$ and a vector $z \in \R^m$, $z\geq 0$ such that for any $x \geq 0$, it holds that $l(x) \leq z$ implies $l(M^{-1}(a+b(x,x))) \leq z$.
\end{description}
We first prove a lemma on weakly positive maps, and then our existence result.

\begin{lemma}\label{wpl}
Let $l:\R^n \to \R^m$ be weakly positive. For each $z\in\R^m$ such that $z\geq 0$ the set $\{x \in \R^n: x\geq 0,\, l(x)\leq z\}$ is bounded.
\end{lemma}
\begin{proof}
 For each $i=1,2,\dots,n$, let $e_i$ denote the $i$-th vector of the canonical basis. The vector $y=l(e_1)$ has at least a nonzero component, let it be $y_j>0$. Then, for each $x\geq 0$ such that $l(x)\leq z$, we must have $x_1\leq \frac{z_j}{y_j}$, otherwise
\[
 \left(l(x)\right)_j> \left(l\left(\frac{z_j}{y_j}e_1\right)\right)_j \geq z_j,
\]
which contradicts $l(x)\leq z$. We may repeat the argument starting from any $e_k$, $k=2,3,\dots,n$ in place of $e_1$, obtaining a corresponding bound for each of the other entries of $x$.
\end{proof}
\begin{theorem}\label{th:minsol}
 Equation \eqref{qve} has at least one solution if and only if A1 holds. Among its solutions, there is a minimal one.
\end{theorem}
\begin{proof}
Let us consider the iteration
\begin{equation}\label{fp1}
 x_{k+1}=M^{-1}\left(a+b(x_k,x_k)\right),
\end{equation}
starting from $x_0=0$. Since $M$ is a nonsingular M-matrix, we have $x_1=M^{-1}a \geq 0$. It is easy to see by induction that $x_k \leq x_{k+1}$:
\[
 x_{k+1}-x_k = M^{-1}(b(x_k,x_k)-b(x_{k-1},x_{k-1})) \geq 0
\]
since $b$ is nonnegative. We now prove by induction that $l(x_k) \leq z$. The base step is clear: $l(0)=0 \leq z$; the inductive step is simply A1. The sequence $x_k$ is nondecreasing and bounded by Lemma~\ref{wpl}; hence it converges. Its limit $x_\ast$ is a solution to \eqref{qve}.

On the other hand, if \eqref{qve} has a solution $s$, then we may choose $l=I$ and $z=s$; now, $0 \leq x \leq s$ implies $M^{-1}(a+b(x,x)) \leq M^{-1}(a+b(s,s))=s$, thus A1 is satisfied with these choices.

For any solution $s$, we may prove by induction that $x_k \leq s$:
\[
 s-x_{k+1} = M^{-1}(a+b(s,s)-a-b(x_k,x_k)) \geq 0.
\]
Therefore, passing to the limit, $x_\ast \leq s$.
\end{proof}

\paragraph{Taylor expansion}
Let $F(x):=Mx-a-b(x,x)$. Since the equation is quadratic, the following expansion holds.
\begin{equation}\label{taylor}
 F(y)=F(x)+F'_x(y-x)+\frac 12 F''_x(y-x,y-x),
\end{equation}
where $F'_x(w)=Mw-b(x,w)-b(w,x)$ is the (Fr\'echet) derivative of $F$ and $F''_x(w,w)=-2b(w,w) \leq 0$ is its second (Fr\'echet) derivative. Notice that $F''_x$ is nonpositive and does not depend on $x$.

\paragraph{Concrete cases}
We may prove A1 for all the examples E1--E4. E1 is covered by the following observation. 
\begin{lemma}\label{soprasol}
If there is a vector $y \geq 0$  such that $F(y) \geq 0$, then A1 holds, and $x_\ast\leq y$.
\end{lemma}
\begin{proof} In fact, we may take the identity map as $l$ and $y$ as $z$. Clearly $0 \leq x \leq y$ implies $M^{-1}(a+b(x,x)) \leq M^{-1}(a+b(y,y)) \leq y$. It is easy to prove by induction that $x_k\leq y$. \end{proof}

As for E2, it follows from the reasoning in \cite{Lu02} that a solution to the specific problem is $u=Xq+e$, $v=X^Tq+e$, where $X$ is the solution of an equation of the form E3; therefore, E2 follows from E3 and Lemma~\ref{soprasol}. An explicit but rather complicate bound to the solution is given in \cite{Juang}.

The case E3 is treated in \cite[Theorem 3.1]{GuoNAREandWienerHopf}. Since $\mathcal M$ in \eqref{nareM} is a nonsingular or singular irreducible M-matrix, there are vectors $v_1,v_2>0$ and $u_1,u_2\geq 0$ such that $Dv_1-Cv_2=u_1$ and $Av_2-Bv_1=u_2$. Let us set $l(x)=X v_1$ and $z=v_2-A^{-1}u_2$.
We have
\[
\begin{aligned}
(AX_{k+1}+X_{k+1}D)v_1 =& (X_kCX_k+B)v_1\\
 \leq&X_kCv_2+Av_2-u_2 \leq X_kDv_1+Av_2-u_2.
\end{aligned}
\]
Since $X_{k+1}Dv_1 \geq X_kDv_1$ (monotonicity of the iteration), we get $X_{k+1}v_1 \leq v_2-A^{-1}u_2$, which is the desired result.

The case E4 is similar. It suffices to set $l(x)=\vettop^{-1}(x)e$ and $z=e$:
\[
 X_{k+1}e=(I-B)^{-1}(A+CX_k^2)e \leq (I-B)^{-1}(Ae+Ce) \leq e,
\]
since $(A+C)e=(I-B)e$

\section{Derivative at the minimal solution}
In order to obtain a cleaner induction proof, we state the next result for a class of iterations slightly more general than the one that we actually need.
\begin{theorem}\label{th:specray}
 Let $x_k$ be the sequence generated by a fixed-point iteration of the form
\begin{align*}
 x_{k+1}=&F(x_k)=a+Px_k+b(x_k,x_k),
\end{align*}
with $a\geq 0, P \geq 0, b\geq 0$, and suppose $x_k$ converges monotonically to  $x_\ast>x_0$. Then $\rho(F'_{x_\ast})\leq 1$, where
\[
 F'_x:=P+b(x,\cdot)+b(\cdot,x)
\]
is the Fr\'echet derivative of the iteration map.
\end{theorem}
\begin{proof}
 Let $e_k:=x_\ast-x_k$; we have $e_{k+1}=P_k e_k$, where
\[
 P_k:=P+b(x_\ast,\cdot)+b(\cdot,x_k).
\]
It is a classical result \cite{Krasno} that
\begin{equation}\label{specray}
 \lim \sup \sqrt[k]{\norm{x_\ast-x_k}} \leq \rho(F'_{x_\ast});
\end{equation}
we first prove that equality holds when $P_0 e_0>0$, following the argument in \cite[Theorem~3.2]{GuoLaub}. Since $P_k$ converges monotonically to $F'_{x_k}$, for any $\varepsilon>0$ we may find an integer $\ell$ such that 
\[
 \rho(P_m) \geq \rho(F'_{x_\ast})-\varepsilon, \quad \forall m \geq \ell.
\]
We have 
\[
\begin{aligned}
 \lim\sup\sqrt[k]{\norm{x_\ast-x_k}} =& \lim \sup \sqrt[k]{\norm{P_{k-1}\dots P_l \dots P_0 e_0}}\\
 \geq&\lim \sup \sqrt[k]{\norm{P_l^{k-l} P_0^l e_0}}.\\
\end{aligned}
\]
Since $P_0 e_0 > 0$, $P_0^l e_0 > C_l e$ for a suitable constant $C_l$. Also, $\norm {P_l^{k-l}}=\norm{P_l^{k-l} v_{k,l}}$ for a suitable $v_{k,l} \geq 0$ with $\norm{v_{k,l}}=1$, so
\[
 \begin{aligned}
 \lim\sup\sqrt[k]{\norm{x_\ast-x_k}} \geq& \lim \sup \sqrt[k]{C_l \norm{P_l^{k-l}e}}\\
 \geq &  \lim \sup \sqrt[k]{C_l \norm{P_l^{k-l}v_{k,l}}}\\
 = & \lim \sup \sqrt[k]{C_l \norm{P_l^{k-l}}}\\
  = & \rho(P_l) \geq \rho(\mathcal G_{x_\ast})-\varepsilon.
\end{aligned}
\]
Since $\varepsilon$ is arbitrary, this shows that equality holds in \eqref{specray}.

The case in which $P_0 e_0$ has some zero entries needs additional considerations. We prove the result by induction on the dimension $n$ of the problem. For $n=1$, either $P_0e_0>0$, and thus the proof above holds, or we are in the trivial case $b(u,v)\equiv 0$. Let us prove the general result in the case in which $P_0 e_0$ has some zero entries.
Suppose that (up to a permutation of the entries)
\begin{align*}
 P_0 e_0=&\vett{t}{0}, & t>0.
\end{align*}
Partition conformably
\begin{align*}
 a=&\vett{a_1}{a_2}, & x_\ast=&\vett{x_{\ast1}}{x_{\ast2}}.
\end{align*}
As $e_1=P_0e_0$ is the error $x_\ast-x_1$, the second block row of $x_k$ needs only one iteration to converge, i.e., for all $k\geq 1$,
\[
 x_k=\vett{t_k}{x_{\ast2}}
\]
for a suitable sequence $t_k$. Moreover, as $P_0\geq 0$ and $e_0>0$, for $P_0e_0$ to have null components we need a special zero structure in $P$ and $b$, namely
\begin{align*}
 P=&\twotwo{P_{1,1}}{P_{1,2}}{0}{0},& b(x_\ast,\cdot)=&\twotwo{B_{1,1}}{B_{1,2}}{0}{0}, & 
b(\cdot,x_\ast)=&\twotwo{C_{1,1}}{C_{1,2}}{0}{0}.
\end{align*}
Therefore, 
\begin{align*}
 F'_{x_\ast}=\twotwo{P_{1,1}+B_{1,1}+C_{1,1}}{P_{1,2}+B_{1,2}+C_{1,2}}{0}{0},
\end{align*}
and our thesis is equivalent to $\rho(P_{1,1}+B_{1,1}+C_{1,1})\leq 1$. The sequence $t_k$, for $k\geq 1$, converges monotonically to $t_\ast=x_{\ast1}>t_1$ and is generated by the fixed-point iteration
\[
 t_{k+1}=a_1+P_{1,1}t_k+P_{1,2}x_{\ast2}+\begin{bmatrix}I & 0\end{bmatrix}b\left(\vett{t_k}{x_{\ast2}},\vett{t_k}{x_{\ast2}}\right)
\]
whose Fr\'echet derivative at the limit point $t_\ast=x_{\ast1}$ is precisely $P_{1,1}+B_{1,1}+C_{1,1}$. Therefore our claim holds by the inductive hypothesis.
\end{proof}
\begin{corollary}
 By applying the theorem above to the fixed-point iteration \eqref{fp1}, we obtain that for a QVE with $x_\ast>0$ $\rho(M^{-1}(b(x_\ast,\cdot)+b(\cdot,x_\ast))\leq 1$ and so $F'_{x_\ast}$ is an $M$-matrix.
\end{corollary}
\begin{corollary}\label{belownonsing}
 If $x_\ast>0$ and $F'_{x_\ast}$ is irreducible or nonsingular, then by \autoref{twoconditions} $F'_{x}$ is a nonsingular M-matrix for all $0 \leq x\leq x_\ast$, $x\neq x_\ast$.
\end{corollary}

\paragraph{Concrete cases}
For E1, only the nonsingular case is of practical interest, thus the results are easier to prove. A strategy to reduce a problem with reducible $F'_{x_\ast}$ to two smaller ones is presented in \cite{williamsburg}. Positivity of the solution and irreducibility are clear for E2 due to the form of the problem. Positivity of the solution has been proved for E3 in \cite{refereeG01} in the case when $M$ is irreducible. Earlier versions of the results appearing in this paper \cite{mythesis,thisoldarxiv} contained an incorrect proof which failed to consider possible zero entries in $P_0e_0$.

\section{Functional iterations}
\subsection{Definition and convergence}
We may define a functional iteration for \eqref{qve} by choosing a splitting $b=b_1+b_2$ such that $b_i \geq 0$ and a splitting $M=N-P$ such that $N$ is an $M$-matrix and $P \geq 0$. We then have the iteration
\begin{equation}\label{funit}
 (N-b_1(\cdot,x_k))x_{k+1}=a+Px_k+b_2(x_k,x_k).
\end{equation}
\begin{theorem}
Suppose that the condition in Corollary~\ref{belownonsing} holds. Let $x_0$ be such that $0 \leq x_0 \leq x_\ast$ and $F(x_0) \leq 0$ (e.g. $x_0=0$). Then:
\begin{enumerate}
 \item $N-b_1(\cdot,x_k)$ is nonsingular for all $k$, i.e., the iteration \eqref{funit} is well-defined.
 \item $x_k\leq x_{k+1}\leq x_\ast$, and $x_k \to x_\ast$ as $k \to \infty$.
 \item $F(x_k)\leq 0$ for all $k$.
\end{enumerate}
\end{theorem}
\begin{proof}
Let $J(x):=N-b_1(\cdot,x)$ and $g(x):=a+Px+b_2(x,x)$. It is clear from the nonnegativity constraints that $J$ is nonincreasing (i.e., $x \leq y \Rightarrow J(x) \geq J(y)$) and $g$ is nondecreasing (i.e., $0 \leq x \leq y \Rightarrow g(x) \leq g(y)$). Furthermore, $J(x)$ is a $Z$-matrix for all $x \geq 0$ and $J(x)\geq F'_x$. Under our assumptions, these results imply that $J(x)$ is a nonsingular $M$-matrix for all $0\leq x\leq x_\ast$, $x\neq x_\ast$, by Corollary~\ref{belownonsing}.

We shall first prove by induction that $x_k\leq x_\ast$. This shows that the iteration is well-posed, since it implies that $J(x_k)$ is a nonsingular $M$-matrix for all $k$. Since $g(x_\ast)=J(x_\ast)x_\ast \leq J(x_k)x_\ast$ by inductive hypothesis, \eqref{funit} implies
\[
J(x_k)(x_\ast-x_{k+1}) \geq g(x_\ast)-g(x_k) \geq 0, 
\]
thus, since $J(x_k)$ is a nonsingular $M$-matrix by inductive hypothesis, $x_\ast-x_{k+1} \geq 0$.

We now prove by induction that $x_k \leq x_{k+1}$. For the base step, since we have $F(x_0) \leq 0$, and $J(x_0)x_0-g(x_0)\leq 0$, thus $x_1=J(x_0)^{-1}g(x_0) \geq x_0$. For $k \geq 1$, 
\[
 J(x_{k-1})(x_{k+1}-x_k) \geq J(x_k)x_{k+1}-J(x_{k-1})x_k=g(x_k)-g(x_{k-1})\geq 0.
\]
thus $x_k\leq x_{k+1}$. 
The sequence $x_k$ is monotonic and bounded above by $x_\ast$, thus it converges. Let $x$ be its limit; by passing \eqref{funit} to the limit, we see that $x$ is a solution. But since $x\leq x_\ast$ and $x_\ast$ is minimal, it must be the case that $x=x_\ast$.

Finally, for each $k \geq 1$ we have
\[
 F(x_k)=J(x_k)x_k-g(x_k) \leq J(x_{k-1})x_{k}-g(x_{k-1})=0. \qedhere
\]
\end{proof}
\begin{theorem}\label{optfunit}
Let $f$ be the map defining the functional iteration \eqref{funit}, i.e. $f(x_k)=J(x_k)^{-1}g(x_k)=x_{k+1}$. Let $\hat J$, $\hat g$, $\hat f$ be the same maps as $J$, $g$, $f$ but for the special choice $b_2=0$, $P=0$. Then $\hat f^k(x) \geq f^k(x)$, i.e., the functional iteration with $b_2=0$, $P=0$ has the fastest convergence among all those defined by \eqref{funit}.
\end{theorem}
\begin{proof}
For each splitting, the functional iteration $f(x)$ is monotonic, i.e., $f(x)\leq f(y)$ whenever $0\leq x\leq y$. Therefore, it suffices to prove that $f(x)\leq \hat{f}(x)$. We have
\begin{align*}
 x-f(x)=&x-J(x)^{-1}g(x)=J(x)^{-1}(J(x)x-g(x))=J(x)^{-1}F(x)\\
\geq& \hat{J}(x)^{-1}F(x)=\hat J(x)^{-1}(\hat J(x)x-\hat g(x))=x-\hat f(x),
\end{align*}
which shows our claim.
\end{proof}
\begin{corollary}
Let 
\begin{equation}\label{gs}
x^{GS}_{k+1}=J(y_k)^{-1}g_k,
\end{equation}
where $y_k$ is a vector such that $x_k \leq y_k \leq x_{k+1}$, and $g_k$ a vector such that $g(x_k) \leq g_k \leq g(x_{k+1})$. It can be proved with the same arguments that $x_{k+1} \leq x^{GS}_{k+1}\leq x_\ast$. This implies that we can perform the iteration in a ``Gauss-Seidel'' fashion: if in some place along the computation an entry of $x_k$ is needed, and we have already computed the same entry of $x_{k+1}$, we can use that entry instead. It can be easily shown that $J(x_k)^{-1}g(x_k) \leq J(y_k)^{-1}g_k$, therefore the Gauss-Seidel version of the iteration converges faster than the original one.
\end{corollary}
\begin{remark}
The iteration \eqref{funit} depends on $b$ as a bilinear form, while Equation \eqref{qve} and its solution depend only on $b$ as a quadratic form. Therefore, different choices of the bilinear form $b$ lead to different functional iterations for the same equation. Since for each iterate of each functional iteration both $x_k\leq x_\ast$ and $F(x_k)\leq 0$ hold (thus $x_k$ is a valid starting point for a new functional iteration), we may safely switch between different functional iterations at every step.
\end{remark}
\paragraph{Concrete cases}
For E1, the algorithm called \emph{depth} in \cite{BeanKontoleonTaylor} is given by choosing $P=0, b_1=0$. The algorithm called \emph{order} in the same paper is obtained in two variants with $P=0, b_2=0$, either on the original problem or on the one with bilinear form $\tilde b(x,y):=b(y,x)$. The algorithm called \emph{thicknesses} in \cite{HautphenneLatoucheRemiche} is given by performing alternately one iteration of each of the two above methods.

For E2, Lu's simple iteration \cite{Lu02} and the algorithm NBJ in \cite{GaoBaiLu} can be seen as the basic iteration \eqref{fp1} and the iteration \eqref{funit} with $P=0, b_2=0$, respectively. The algorithm NBGS in the same paper is a Gauss-Seidel-like variant. It is shown in \cite{refereeGL10} that NBGS is twice as fast as NBJ in terms of asymptotic rate of convergence.

For E3, the fixed point iterations in \cite{GuoLaub} are given by $b_2=b$ and different choices of $P$. The iterations in \cite{JuangChen} are the one given by $b_2=0, P=0$ and a Gauss-Seidel-like variant.

For E4, the iterations in \cite[chapter 6]{BiniLatoucheMeiniBook} can also be reinterpreted in our framework.
\section{Newton's method}

\subsection{Definition and convergence}
We may define the Newton method for the equation \eqref{qve} as
\begin{equation}\label{newton}
 F'_{x_k} (x_{k+1}-x_k) = -F(x_k).
\end{equation}
Alternatively, we may write
\begin{equation}\label{fxkbuffo}
 F'_{x_k}(x_{k+1})=a-b(x_k,x_k).
\end{equation}
Also notice that
\begin{equation}\label{altrofxk}
 -F(x_{k+1})=b(x_{k+1}-x_k,x_{k+1}-x_k).
\end{equation}
\begin{theorem}
Suppose that the condition in Corollary~\ref{belownonsing} holds. The Newton method \eqref{newton} starting from $x_0=0$ is well-defined, and the generated sequence $x_k$ converges monotonically to $x_\ast$.
\end{theorem}
\begin{proof}
We shall prove by induction that $x_k \leq x_{k+1} \leq x_\ast$. We have $x_1=M^{-1}a\geq 0$ and $x_\ast=M^{-1}a+M^{-1}b(x_\ast,x_\ast)\geq M^{-1}a$, so the base step holds. From \eqref{altrofxk}, we get
\[
 F'_{x_{k+1}}(x_{k+2}-x_{k+1}) = b(x_{k+1}-x_k,x_{k+1}-x_k) \geq 0,
\]
thus, since $F'_{x_{k+1}}$ is a nonsingular M-matrix, $x_{k+2}\geq x_{k+1}$. Similarly, from \eqref{fxkbuffo},
\[
 F'_{x_{k+1}}(x_\ast-x_{k+2})=b(x_\ast-x_{k+1},x_\ast-x_{k+1}),
\]
thus $x_{k+2}\leq x_\ast$.
The sequence $x_k$ is monotonic and bounded from above by $x_\ast$, thus it converges; by passing \eqref{newton} to the limit we see that its limit must be a solution of $\eqref{qve}$, hence $x_\ast$.
\end{proof}

\subsection{Concrete cases}
Newton methods for E1 and E2 appear respectively in \cite{HautphenneLatoucheRemiche} and \cite{LuNewton}. Monotonic convergence of the Newton method for E3 has originally been proved with the additional hypothesis $x_1>0$ in \cite{GuoLaub} and \cite{GuoNAREandWienerHopf}, but this assumption was later removed in \cite{GuoH07}. For E4, the Newton method is described in a more general setting in \cite{BiniLatoucheMeiniBook,latouchenewton}, and can be implemented using the method described in \cite{amatoetal} for the solution of the resulting Sylvester equation. However, different methods such as cyclic and logarithmic reduction \cite{BiniLatoucheMeiniBook} are usually preferred due to their lower computational cost.

\section{Modified Newton method}\label{s:mn}
Recently Hautphenne and Van Houdt \cite{HautphenneVanHoudt} proposed a different version of Newton's method for E1 that has a better convergence rate than the traditional one. Their idea is to apply the Newton method to the equation
\begin{equation}\label{hyper}
 G(x)=x-(M-b(\cdot,x))^{-1}a,
\end{equation}
which is equivalent to \eqref{qve}.

\subsection{Theoretical properties}
Let us set for the sake of brevity $R_x:=M-b(\cdot,x)$. When the condition in Corollary~\ref{belownonsing} holds, $R_x\geq F'_x$ is nonsingular for every $x\leq x_\ast$.
 The Jacobian of $G$ is
\[
 G'_x=I-R_x^{-1}b(R_x^{-1}a,\cdot).
\]
As for the original Newton method, it is a $Z$-matrix, and a nonincreasing function of  $x$. It is easily seen that $G'_{x_\ast}$ is an $M$-matrix. The proof in Hautphenne and Van Houdt \cite{HautphenneVanHoudt} is of probabilistic nature and cannot be extended to our setting; we shall provide here a different one. We have
\[
 G'_x \geq G'_{x_\ast}=R_{x_{\ast}}^{-1}\left(M-b(\cdot,x_\ast)-b(R_{x_\ast}^{-1}a,\cdot)\right)=
 R_{x_\ast}^{-1}\left(M-b(\cdot,x_{\ast})-b(x_{\ast},\cdot)\right)=R_{x_\ast}^{-1}F'_{x_\ast};
\]
Thus when the condition in Corollary~\ref{belownonsing} holds, $G'_x$ is a nonsingular M-matrix and thus the modified Newton method is well-defined. The monotonic convergence is easily proved in the same fashion as for the traditional method.

The following result holds.
\begin{theorem}[\cite{HautphenneVanHoudt}]
Let $\tilde x_k$ be the iterates of the modified Newton method and $x_k$ those of the traditional Newton method, starting from $\tilde x_k=x_k=0$. Then $\tilde x_k-x_k \geq 0$.
\end{theorem}
The proof in Hautphenne and Van Houdt \cite{HautphenneVanHoudt} can be adapted to our setting with minor modifications.

\subsection{Concrete cases}

Other than for E1, its original setting, the modified Newton method is useful for the other concrete cases of quadratic vector equations.

For E2, let us choose the bilinear map $b$ as
\[
 b\left(\vett{u_1}{v_1},\vett{u_2}{v_2}\right):= \vett{u_1 \circ (Pv_2)}{v_1 \circ (\tilde P u_2)}.
\]
This way, it is easily seen that $b(\cdot,x)$ is a diagonal matrix and $b(x,\cdot)$ has the same structure that allowed a fast (with $O(n^2)$ operations per step) implementation of the traditional Newton's method in Bini \emph{et al.} \cite{BiniIannazzoPoloni}. Therefore the modified Newton method can be implemented with a negligible overhead ($O(n)$ ops per step on an algorithm that takes $O(n^2)$ ops per step) with respect to the traditional one, and increased convergence rate.

We have performed some numerical experiments on the modified Newton method for E2; as can be seen in \autoref{fig:mn}, the modified Newton method does indeed converge faster to the minimal solution, and this allows one to get better approximations to the solution with the same number of steps.
\begin{figure}[t]
\includegraphics{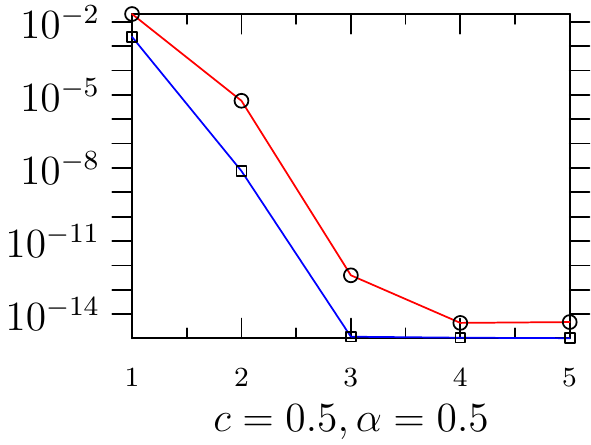}
\phantom{ }
\includegraphics{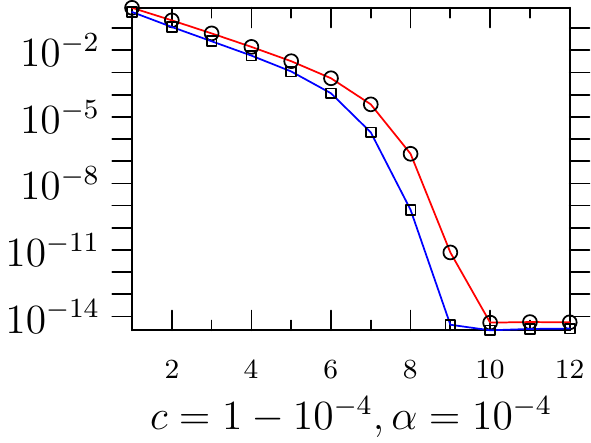}

\includegraphics{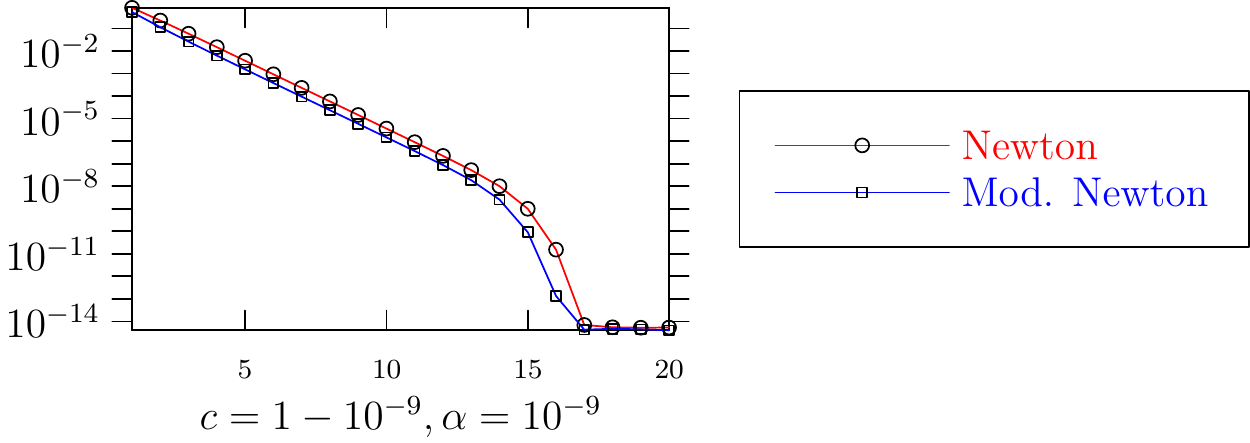}

 \caption{Convergence history of the two Newton methods for E2 for several values of the parameters $\alpha$ and $c$. The plots show the residual Frobenius norm of Equation \eqref{qve} vs. the number of iterations}\label{fig:mn}
\end{figure}

For E3 and E4, the modified Newton method leads to similar equations to the traditional one (continuous- and discrete-time Sylvester equations), but requires additional inversions and products of $m \times m$ matrices; that is, the overhead is of the same order of magnitude $O(m^3)$ of the cost of the Newton step. Therefore it is not clear whether the improved convergence rate makes up for the increase in the computational cost.

\section{Newton method and Cyclic/Logarithmic Reduction}\label{s:nvscr}
\subsection{Recall of Logarithmic and Cyclic Reduction}
Cyclic and Logarithmic Reduction \cite{BiniLatoucheMeiniBook} are two closely related methods for solving E4, which have quadratic convergence and a lower computational cost than Newton's method. Both are based on specific properties of the problem and cannot be extended in a straightforward way to any quadratic vector equation.

Logarithmic Reduction (LR) is based on the fact that if $X$ solves
\[
 X=B_{-1}+B_1 X^2,
\]
then it can be shown with algebraic manipulations that $Y=X^2$ solves the equation
\begin{equation}\label{lr2}
 Y=(I-B_{-1}B_1-B_1B_{-1})^{-1} \left(B_{-1}\right)^2 + (I-B_{-1}B_1-B_1B_{-1})^{-1} \left(B_{1}\right)^2 Y^2,
\end{equation}
with the same structure. Therefore we may start from an approximation $X_1=B_{-1}$ to the solution and refine it with a term $B_1 Y$, where $Y$ is (an approximation to) the solution of \eqref{lr2}. Such an approximation is computed with the same method, and refined successively by applying the same method recursively. The resulting algorithm is reported here as Algorithm~\ref{algo:lr}.
\begin{algorithm}
  \KwIn{$A$, $B$, $C$}
  \KwOut{minimal solution $X$ of $X=A+BX+CX^2$}
  $B_{-1}\leftarrow(I-B)^{-1}A$\;
  $B_1\leftarrow (I-B)^{-1}C$\;
  $X\leftarrow B_{-1}$\;
  $U\leftarrow B_1$\;
  \While{stopping criterion is not satisfied}{
    $C \leftarrow I-B_1B_{-1}-B_{-1}B_1$\;
    $B_{-1} \leftarrow C^{-1}B_{-1}^2$\;
    $B_{1} \leftarrow C^{-1}B_{1}^2$\;
    $X\leftarrow X+UB_{-1}$\;
    $U\leftarrow UB_1$\;
  }
  \KwRet{$X$}
\caption{Logarithmic Reduction for E4 \cite{BiniLatoucheMeiniBook}}\label{algo:lr}
\end{algorithm}
An alternative interpretation of LR \cite{BiniLatoucheMeiniBook} arises by defining the matrix-valued function $f: \mathbb C \to \mathbb C^{m \times m}$ as $f(z)=B_{-1}-z+B_1z^2$ and applying the Graeffe iteration $f \mapsto f(z)f(-z)$, which yields a quadratic polynomial in $z^2$ with the same roots of $f(z)$ plus some additional ones.

Cyclic Reduction (CR) is a similar algorithm, which is connected to LR by simple algebraic relations (see the Bini \emph{et al.} book \cite{BiniLatoucheMeiniBook} for more detail). We shall report it here as Algorithm~\ref{algo:cr}.
\begin{algorithm}
  \KwIn{$A$, $B$, $C$}
  \KwIn{minimal solution $X$ of $X=A+BX+CX^2$}
  $R\leftarrow I-B$\;
  $S\leftarrow I-B$\;
  $A_0 \leftarrow A$\;
  \While{stopping criterion is not satisfied}{    
    $S\leftarrow R-CR^{-1}A$\;
    $X\leftarrow S^{-1}A_0$\;
    $R'\leftarrow R-AR^{-1}C-CR^{-1}A$\;
    $A'\leftarrow AR^{-1}A$\;
    $C'\leftarrow CR^{-1}C$\;
    $R,A,C \leftarrow R',A',C'$\;
  }
  \KwRet{$X$}
\caption{Cyclic Reduction for E4 \cite{BiniLatoucheMeiniBook}}\label{algo:cr}
\end{algorithm}

\subsection{Generalization attempts}
We may attempt to produce algorithms similar to LR and CR for a generic quadratic vector equation. Notice that we cannot look for an equation in $X^2$ in our vector setting, since $x^2$ for a vector $x$ has not a clear definition --- using e.g. the Hadamard (component-wise) product does not lead to a simple equation. Nevertheless, we may try to find an equation in $b(x,x)$, which is the only quadratic expression that makes sense in our context.

We look for an expression similar to the Graeffe iteration. If $x$ solves $0=F(x)=Mx-a-b(x,x)$, then it also solves $b(F(x),F(-x))+b(F(-x),F(x))=0$ (notice that a symmetrization is needed), that is,
\[
\begin{aligned}
 b(x-M^{-1}a-M^{-1}b(x,x) , x+M^{-1}a+M^{-1}b(x,x))+ \\
b(x+M^{-1}a+M^{-1}b(x,x) , x-M^{-1}a-M^{-1}b(x,x))=0.
\end{aligned}
\]
If we set $v_1=M^{-1}b(x,x)$ and exploit the bilinearity of $b(\cdot,\cdot)$, the above equation reduces to
\begin{equation}\label{eq2}
 -b(M^{-1}a,M^{-1}a) + \left(M-b(M^{-1}a,\cdot)-b(\cdot,M^{-1}a)\right)v_1 -b(v_1,v_1)=0,
\end{equation}
which is suitable to applying the same process again. A first approximation to $x$ is given by $M^{-1}a$; if we manage to solve (even approximately) \eqref{eq2}, this approximation can be refined as $x=M^{-1}a+v_1$. We may apply this process recursively, getting an algorithm similar to Logarithmic Reduction. The algorithm is reported here as Algorithm~\ref{algo:nr}.
\begin{algorithm}
 \KwIn{$a$, $M$, $b$}
 \KwOut{minimal solution $x$ of $Mx=a+b(x,x)$}
  $x \leftarrow 0$, $\tilde M \leftarrow M$, $\tilde a \leftarrow a$\;
  \While{stopping criterion is not satisfied}{
  $w \leftarrow \tilde M^{-1} \tilde a$\;
  $x \leftarrow x+w$\;
  $\tilde a \leftarrow b(w,w)$\;
  $ \tilde M \leftarrow  \tilde M-b(w,\cdot)-b(\cdot,w)$\;
  }
  \KwRet{$x$}
 \caption{A Cyclic Reduction-like formulation of Newton's method for a quadratic vector equation}\label{algo:nr}
\end{algorithm}
It is surprising to see that this algorithm turns out to be equivalent to Newton's method. In fact, it is easy to prove by induction the following proposition.
\begin{theorem}
 Let $x_k$ be the iterates of Newton's method on \eqref{qve} starting from $x_0=0$.
 At the $k$th iteration of the \textbf{while} cycle in Algorithm~\ref{algo:nr},
 $x=x_k$, $w=x_{k+1}-x_k$, $\tilde M=F'_{x_k}$, $\tilde a=-F(x_k)$.
\end{theorem}

The modified Newton method discussed in \autoref{s:mn} can also be expressed in a form that looks very similar to LR/CR. We may express all the computations of step $k+1$ in terms of $R_{x_k}^{-1}b$ and $R_{x_k}^{-1}a$ only: in fact,
\[
 R_{x_{k+1}}=R_{x_k}-b(\cdot,x_{k+1}-x_k)=R_{x_k}\left(I-R_{x_k}^{-1}b(\cdot,x_{k+1}-x_k) \right),
\]
and thus
\[
 R_{x_{k+1}}^{-1}R_{x_k}=\left(I-R_{x_k}^{-1}b(\cdot,x_{k+1}-x_k)  \right)^{-1}.
\]
The resulting algorithm is reported here as Algorithm~\ref{algo:mnr}.
\begin{algorithm}
 \KwIn{$a$, $M$, $b$}
 \KwOut{minimal solution $x$ of $Mx=a+b(x,x)$}
  $x \leftarrow 0$, $\tilde a \leftarrow a$, $\tilde b \leftarrow b$, $\tilde w \leftarrow 0$\;
  \While{stopping criterion is not satisfied}{
  $\tilde a \leftarrow (I-\tilde b(\cdot,w))^{-1}\tilde a$\;
  $\tilde b \leftarrow (I-\tilde b(\cdot,w))^{-1}\tilde b$\;
  $w \leftarrow (I-\tilde b(\tilde a,\cdot))^{-1}(\tilde a-x)$\;
  $x \leftarrow x+w$\;
  }
  \KwRet{$x$}
 \caption{A Cyclic Reduction-like formulation of the modified Newton method for a quadratic vector equation}\label{algo:mnr}
\end{algorithm}

The similarities between the two Newton formulations and LR are apparent. In all of them, only two variables ($B_{-1}$ and $B_1$, $\tilde a$ and $\tilde M$, $\tilde a$ and $\tilde b$) are stored and used to carry on the successive iteration, and some extra computations and variables are needed to extract from them the approximation of the solution ($X$, $x$) which is refined at each step with a new additive term.

It is a natural question whether there are algebraic relations among LR and Newton methods, or if LR can be interpreted as an inexact Newton method (see e.g. Ortega and Rheinboldt \cite{OrtegaRheinboldt}), thus providing an alternative proof of its quadratic convergence. However, we were not able to find an explicit relation among the two classes of methods. This is mainly due to the fact that the LR and CR methods are based upon the squaring $X \mapsto X^2$, which we have no means to translate in our vector setting. To this regard we point out that we cannot invert the matrix $C$, since in many applications it is strongly singular.

\section{Positivity of the minimal solution}\label{sec:pos}
\subsection{Role of the positivity}
In many of the above theorems, the hypothesis $x_\ast>0$ is required. Is it really necessary? What happens if it is not satisfied?

In all the algorithms we have exposed, we worked with only vectors $x$ such that $0 \leq x \leq x_\ast$. Thus, if $x_\ast$ has some zero entry, we may safely replace the problem with a smaller one by projecting the problem on the subspace of all vectors that have the same zero pattern as $x_\ast$: i.e., we may replace the problem with the one defined by
\[
 \hat a=\Pi a,\, \hat M= \Pi M \Pi^T,\, \hat b(x,y)=\Pi b(\Pi^T x, \Pi^T y),
\]
where $\Pi$ is the orthogonal projector on the subspace
\begin{equation}\label{defW}
W=\{x \in \R^n : \text{$x_i=0$ for all $i$ such that $(x_{\ast})_i=0$}\},
\end{equation}
i.e. the linear operator that removes the entries known to be zero from the vectors.
Performing the above algorithms on the reduced vectors and matrices is equivalent to performing them on the original versions, provided the matrices to invert are nonsingular. Notice, though, that both functional iterations and Newton-type algorithms may break down when the minimal solution is not strictly positive. For instance, consider the problem
\[
 a=\vett{\frac 12}{0},\,
 M=I_2,\,
 b\left(\vett{x_1}{x_2},\vett{y_1}{y_2}\right)=\vett{\frac 12 x_1 y_1}{Kx_1y_2},\,
 x_\ast=\vett{1}{0}.
\]
For suitable choices of the parameter $K$, the matrices to be inverted in the functional iterations (excluding obviously \eqref{fp1}) and Newton's methods are singular; for large values of $K$, none of them are $M$-matrices. However, the nonsingularity and $M$-matrix properties still hold for their restrictions to the subspace $W$ defined in \eqref{defW}. It is therefore important to consider the positivity pattern of the minimal solution in order to get working algorithms.
\subsection{Computing the positivity pattern}
By considering the functional iteration \eqref{fp1}, we may derive a method to infer the positivity pattern of the minimal solution in time $O(n^3)$. Let us denote by $e_t$ the $t$-th vector of the canonical basis, and $e_S=\sum_{s \in S} e_s$ for any set $S\in\{1,\dots,n\}$.

The main idea of the algorithm is following the iteration $x_{k+1}=M^{-1}(a+b(x_k,x_k))$, checking at each step which entries become (strictly) positive. Since the iterates are nondecreasing, once an entry becomes positive for some $k$ it stays positive. It is possible to reduce substantially the number of operations needed if we follow a different strategy to perform these checks. We consider a set $S$ of entries known to be positive at a certain step $x_k$, i.e., $i\in S$ if we already know that $(x_k)_i>0$ at some step $k$. At the first step of the iteration, only the entries $i$ such that $M^{-1}a_i>0$ belong to this set. For each entry $i$, we check whether we can deduce the positiveness of more entries thanks to $i$ and some other $j\in S$ being positive, using the nonzero pattern of $M^{-1}b(\cdot,\cdot)$. As we prove formally in the following, it suffices to consider each $i$ once in this process. Therefore, we consider a second set $T \subseteq S$ of positive entries that have not been checked for the consequences of their positiveness, and examine them one after the other.

We report the algorithm as Algorithm~\ref{algo:posit}, and proceed to prove that it computes the support of $x_\ast$.
\begin{algorithm}
  \KwIn{$a$, $M$, $b$}
  \KwOut{$S=\{i:(x_{\ast})_i>0\}$}
  $S \leftarrow \emptyset$\tcp*{entries known to be positive}
  $T \leftarrow \emptyset$\tcp*{entries to check}
  $a' \leftarrow M^{-1}a$\;
  \For{$i=1$ to $n$}{
  \lIf{$a'_i>0$}{$T \leftarrow T \cup \{i\}$;  $S \leftarrow S \cup \{i\}$\;}
  }
  \While{$T \neq \emptyset$ and $S\neq \{1,2,\dots,n\}$}{
  $t \leftarrow $ some element of $T$\;
  $T \leftarrow T \setminus \{t\}$\;
  $u \leftarrow M^{-1}\left(b(e_S,e_t)+b(e_t,e_S)\right)$\tcp*{or only its positivity pattern}
  \For{$i \in \{1,\dots,n\} \setminus S$}{
  \lIf{$u_i>0$}{$T \leftarrow T \cup \{i\}$;  $S \leftarrow S \cup \{i\}$\;}}
  }
  \KwRet{$S$}
\caption{Compute the positivity pattern of the solution $x_\ast$}\label{algo:posit}
\end{algorithm}
\begin{theorem}
 The above algorithm runs in at most $O(n^3)$ operations.
\end{theorem}
\begin{proof}
For the implementation of the sets, we shall use the simple approach to keep in memory two vectors $S,T \in \{0,1\}^n$ and set to 1 the components relative to the indices in the sets. With this choice, insertions and membership tests are $O(1)$, loops are easy to implement, and retrieving an element of the set costs at most $O(n)$.

If we precompute a PLU factorization of $M$, each subsequent operation $M^{-1}v$, for $v \in \R^n$, costs $O(n^2)$. The first \textbf{for} loop runs in at most $O(n)$ operations. The body of the \textbf{while} loop runs at most $n$ times, since an element can be inserted into $S$ and $T$ no more than once ($S$ never decreases). Each of its iterations costs $O(n^2)$, since evaluating $b(e_t,e_S)$ is equivalent to computing the matrix-vector product between the matrix $(B_{tij})_{i,j=1,\dots,n}$ and $e_S$, and similarly for $b(e_S,e_t)$.
\end{proof}
The fact that the algorithm computes the right set may not seem obvious at first sight. We prove this result by resorting to an alternative characterization of the positive entries of $x_\ast$. For fixed $a,M,b$, and for a fixed $i\in\{1,2,\dots,n\}$, we call a sequence $\{S_h\}_{h=1}^N$ of subsets of $\{1,2,\dots,n\}$ \emph{positivity-showing for $i$} if it satisfies the following properties
\begin{enumerate}[i.]
 \item $S_1=\{h \in \{1,\dots,n\} : (M^{-1}a)_h>0 \}$;\label{basecase}
 \item $S_h \subset S_{h+1}$ for each $h\geq 1$;\label{increasingness}
 \item for each $t\in S_{h+1}\setminus S_h$, there are $r,s \in S_h$ such that $(M^{-1}B)_{rst}>0$ (possibly $r=s$);\label{connectedness}
 \item $i\in S_N$. \label{last}
\end{enumerate}
\begin{lemma}
 For each $i\in\{1,2,\dots,n\}$, $(x_\ast)_i>0$ if and only if there exist a positivity-showing sequence for $i$.
\end{lemma}
\begin{proof}
\begin{itemize}
 \item[$\Rightarrow$] Take $i$ such that $(x_\ast)_i>0$, and consider the iteration \eqref{fp1}. Since $x_k\to x_\ast$, we must have $(x_N)_i>0$ for sufficiently large $N$. Then, we can prove that the sequence $S_h=\{i : (x_h)_i>0\}$, $h=1,2,\dots,N$ is positivity-showing for $i$. Conditions \ref{basecase} and \ref{last} are clear; \ref{increasingness} is satisfied because the iteration is monotonic, and \ref{connectedness} is satisfied because we need a nonzero summand in the right-hand side of
\begin{equation}\label{fp1explicit}
(x_{h+1})_k = (M^{-1}a)_k + \sum (M^{-1}B)_{ijk} (x_h)_i (x_h)_j
\end{equation}
for the left-hand side to be positive.
\item[$\Leftarrow$] given a positivity-showing sequence, we can prove by induction on $k$ that $(x_h)_k>0$ for each $h\in S_k$, where $x_h$ are again the iterates of \eqref{fp1}. The base step is condition \ref{basecase}, the inductive step follows from the fact that there is at least a nonzero summand in the right-hand side of \eqref{fp1explicit} and thus the left-hand side is positive. In particular, $(x_N)_i>0$ and thus $(x_\ast)_i>0$. \qedhere
\end{itemize}
\end{proof}
\begin{lemma}
The set $S$ returned by Algorithm~\ref{algo:posit} contains $i$ if and only if there is a positivity-showing sequence for $i$.
\end{lemma}
\begin{proof}
\begin{itemize}
\item[$\Rightarrow$] If at some step of the algorithm we have $i\in S$, then the values of $S$ at every previous step of the algorithm form a positivity-showing sequence.
\item[$\Leftarrow$] We prove the result by induction on the length of the shortest positivity-showing sequence for each given $i$. The case $N=1$ is clear, since it must be the case that $(M^{-1}a)_i>0$. Let us now suppose that the result is proved for all $i$ for which the shortest positivity-showing sequence has length $N-1$, and prove the claim for $N$. By condition \ref{connectedness}, there are $r,s \in S_{N-1}$ such that $(M^{-1}B)_{rsi}>0$. The sequence $S_1,S_2,\dots,S_{N-1}$ is a positivity-showing sequence of length $N-1$ for all the elements of $S_{N-1}$, thus by inductive hypothesis $r$ and $s$ enter $S$ (and at the same time $T$) at some step of the algorithm. If the \textbf{while} cycle terminates because $S=\{1,2,\dots,n\}$, there $i\in S$ and there is nothing to prove. Otherwise, the algorithm terminates because $T=\emptyset$, and thus $r$ and $s$ are removed from $T$ at some step after being inserted. In the iteration of the \textbf{while} cycle in which either one of them is removed from $T$, we have $u_i>0$ and thus $i$ enters $S$. \qedhere
\end{itemize}
\end{proof}
The two lemmas proved above together imply that Algorithm~\ref{algo:posit} computes the correct set $S$.

It is a natural question to ask whether for the cases E3 and E4 it is possible to use the special structure of $M$ and $b$ in order to develop a similar algorithm with running time $O(m^3)$, that is, the same as the cost per step of the basic iterations. Unfortunately, we were unable to go below $O(m^4)$. It is therefore much less appealing to run this algorithm as a preprocessing step, since its cost is likely to outweigh the cost of the actual solution. However, we remark that the strict positiveness of the coefficients is usually a property of the problem rather than of the specific matrices involved, and can often be solved in the model phase before turning to the actual computations. An algorithm such as the above one would only be needed in an ``automatic'' subroutine to solve general instances of the problems E3 and E4. 

\section{Other concrete cases}
In Bini \emph{et al.} \cite{BiniLatoucheMeiniTreeLike}, the matrix equation
\[
 X+\sum_{i=1}^d A_iX^{-1}D_i=B-I
\]
appears, where $B,A_i,D_i \geq 0$ and the matrices $B+D_j+\sum_{i=1}^d A_i$ are stochastic. The solution $X=T-I$, with $T \geq 0$ minimal and sub-stochastic, is sought. Their paper proposes a functional iteration and Newton's method. By setting $Y=-X^{-1}$ and multiplying both sides by $Y$, we get
\[
 (I-B)Y=I+\sum A_i Y D_i Y,
\]
which is again in the form \eqref{qve}. It is easy to see that $Y$ is nonnegative whenever $T$ is substochastic, and $Y$ is minimal whenever $T$ is.

The paper considers two functional iterations and the Newton method; all these algorithms are expressed in terms of $X$ instead of $Y$, but they essentially coincide with those exposed in the present paper.

\section{Research lines}
There are many open questions that could yield a better theoretical understanding of this class of equations or better solution algorithms.
\begin{itemize}
 \item Is there a way to translate to our setting the spectral theory of E4 (see e.g. Bini \emph{et al.} \cite[chapter 3]{BiniLatoucheMeiniBook})?
 \item The shift technique \cite[chapter 3]{BiniLatoucheMeiniBook} is a method to transform a singular problem (i.e. one in which $F'_{x_\ast}$ is singular) of the kind E4 (or also E3, see e.g. \cite{GuoQBDRed,BiniIannazzoLatoucheMeini}) to a nonsingular one. Is there a way to adapt it to a generic quadratic vector equation? Is there a similar technique for near-to-singular problems, which are the most difficult to solve in the applications?
 \item As we discussed in the \autoref{s:nvscr}: is there an explicit algebraic relation among Newton's method and Logarithmic/Cyclic Reduction, or an interpretation of the latter as an inexact Newton method?
 \item Instead of \eqref{funit}, one could consider the slightly more general form
\[
 (N-b_1(\cdot,x_k)-b_3(x_k,\cdot))x_{k+1}=a+Px_k+b_2(x_k,x_k),
\]
where $b=b_1+b_2+b_3$ and $M=N-P$. This notation would incorporate the two variants of the \emph{order} algorithm at the same time. We can prove as in Theorem~\ref{optfunit} that $P=b_2=0$ is the best choice (among those with $P,b_1,b_2,b_3\geq 0$), but it is not clear how to determine \emph{a priori} the choice of $b_1$ and $b_3$ which gives the fastest convergence. Also, is there an explicit relation between the \emph{thicknesses} method of Hautphenne \emph{et al.} \cite{HautphenneLatoucheRemiche} and the symmetrized functional iteration given by $b_1=b_3=\frac 12 b$?
 \item Can this approach be generalized to the positive definite ordering on symmetric matrices ($A \geq B$ if $A-B$ is positive semidefinite)? This would lead to the further unification of the theory of a large class of equations, including the algebraic Riccati equations appearing in control theory \cite{LancasterRodman}. A lemma proved by Ran and Reurings \cite[theorem 2.2]{RanReurings} could replace the first point of \autoref{Mmatrices} in an extension of the results of this paper to the positive definite ordering.
 \end{itemize}
 
\section{Acknowledgment}
The author wishes to show his gratitude to an anonymous referee for identifying some problematic issues with the first version of this paper, suggesting the proof of item~\ref{Mmatrices5} of Theorem~\ref{Mmatrices} and providing additional literature references.

\bibliographystyle{abbrv}
\bibliography{qve}

\def\cprime{$'$}
\begin{thebibliography}{10}

\bibitem{GaoBaiLu}
Z.-Z. Bai, Y.-H. Gao, and L.-Z. Lu.
\newblock Fast iterative schemes for nonsymmetric algebraic {R}iccati equations
  arising from transport theory.
\newblock {\em SIAM J. Sci. Comput.}, 30(2):804--818, 2008.

\bibitem{BeanKontoleonTaylor}
N.~G. Bean, N.~Kontoleon, and P.~G. Taylor.
\newblock Markovian trees: properties and algorithms.
\newblock {\em Ann. Oper. Res.}, 160:31--50, 2008.

\bibitem{BermanPlemmons}
A.~Berman and R.~J. Plemmons.
\newblock {\em Nonnegative matrices in the mathematical sciences}, volume~9 of
  {\em Classics in Applied Mathematics}.
\newblock Society for Industrial and Applied Mathematics (SIAM), Philadelphia,
  PA, 1994.
\newblock Revised reprint of the 1979 original.

\bibitem{BiniIannazzoLatoucheMeini}
D.~A. Bini, B.~Iannazzo, G.~Latouche, and B.~Meini.
\newblock On the solution of algebraic {R}iccati equations arising in fluid
  queues.
\newblock {\em Linear Algebra Appl.}, 413(2-3):474--494, 2006.

\bibitem{BiniIannazzoPoloni}
D.~A. Bini, B.~Iannazzo, and F.~Poloni.
\newblock A fast {N}ewton's method for a nonsymmetric algebraic {R}iccati
  equation.
\newblock {\em SIAM J. Matrix Anal. Appl.}, 30(1):276--290, 2008.

\bibitem{BiniLatoucheMeiniTreeLike}
D.~A. Bini, G.~Latouche, and B.~Meini.
\newblock Solving nonlinear matrix equations arising in tree-like stochastic
  processes.
\newblock {\em Linear Algebra Appl.}, 366:39--64, 2003.
\newblock Special issue on structured matrices: analysis, algorithms and
  applications (Cortona, 2000).

\bibitem{BiniLatoucheMeiniBook}
D.~A. Bini, G.~Latouche, and B.~Meini.
\newblock {\em Numerical methods for structured {M}arkov chains}.
\newblock Numerical Mathematics and Scientific Computation. Oxford University
  Press, New York, 2005.
\newblock Oxford Science Publications.

\bibitem{williamsburg}
D.~A. Bini, B.~Meini, and F.~Poloni.
\newblock On the solution of a quadratic vector equation arising in markovian
  binary trees, 2010.
\newblock arXiv:1011.1233. Available at \url{http://arxiv.org/abs/1011.1233}.

\bibitem{amatoetal}
J.~D. Gardiner, A.~J. Laub, J.~J. Amato, and C.~B. Moler.
\newblock Solution of the {S}ylvester matrix equation {$AXB^T+CXD^T=E$}.
\newblock {\em ACM Trans. Math. Software}, 18(2):223--231, 1992.

\bibitem{GuoNAREandWienerHopf}
C.-H. Guo.
\newblock Nonsymmetric algebraic {R}iccati equations and {W}iener-{H}opf
  factorization for {$M$}-matrices.
\newblock {\em SIAM J. Matrix Anal. Appl.}, 23(1):225--242 (electronic), 2001.

\bibitem{refereeG01}
C.-H. Guo.
\newblock A note on the minimal nonnegative solution of a nonsymmetric
  algebraic {R}iccati equation.
\newblock {\em Linear Algebra Appl.}, 357:299--302, 2002.

\bibitem{GuoQBDRed}
C.-H. Guo.
\newblock Efficient methods for solving a nonsymmetric algebraic {R}iccati
  equation arising in stochastic fluid models.
\newblock {\em J. Comput. Appl. Math.}, 192(2):353--373, 2006.

\bibitem{GuoH07}
C.-H. Guo and N.~J. Higham.
\newblock Iterative solution of a nonsymmetric algebraic {R}iccati equation.
\newblock {\em SIAM J. Matrix Anal. Appl.}, 29(2):396--412, 2007.

\bibitem{GuoLaub}
C.-H. Guo and A.~J. Laub.
\newblock On the iterative solution of a class of nonsymmetric algebraic
  {R}iccati equations.
\newblock {\em SIAM J. Matrix Anal. Appl.}, 22(2):376--391 (electronic), 2000.

\bibitem{refereeGL10}
C.-H. Guo and W.-W. Lin.
\newblock Convergence rates of some iterative methods for nonsymmetric
  algebraic {R}iccati equations arising in transport theory.
\newblock {\em Linear Algebra Appl.}, 432(1):283--291, 2010.

\bibitem{HautphenneLatoucheRemiche}
S.~Hautphenne, G.~Latouche, and M.-A. Remiche.
\newblock Newton's iteration for the extinction probability of a {M}arkovian
  binary tree.
\newblock {\em Linear Algebra Appl.}, 428(11-12):2791--2804, 2008.

\bibitem{HautphenneVanHoudt}
S.~Hautphenne and B.~Van~Houdt.
\newblock On the link between {M}arkovian trees and tree-structured {M}arkov
  chains.
\newblock {\em Europ. J. Op. Res.}, 2009.
\newblock doi:10.1016/j.ejor.2009.03.052. Article in press.

\bibitem{Juang}
J.~Juang.
\newblock Global existence and stability of solutions of matrix {R}iccati
  equations.
\newblock {\em J. Math. Anal. Appl.}, 258(1):1--12, 2001.

\bibitem{JuangChen}
J.~Juang and I.~D. Chen.
\newblock Iterative solution for a certain class of algebraic matrix {R}iccati
  equations arising in transport theory.
\newblock {\em Transport Theory Statist. Phys.}, 22(1):65--80, 1993.

\bibitem{Krasno}
M.~A. Krasnosel{\cprime}ski{\u\i}, G.~M. Va{\u\i}nikko, P.~P. Zabre{\u\i}ko,
  Y.~B. Rutitskii, and V.~Y. Stetsenko.
\newblock {\em Approximate solution of operator equations}.
\newblock Wolters-Noordhoff Publishing, Groningen, 1972.
\newblock Translated from the Russian by D. Louvish.

\bibitem{LancasterRodman}
P.~Lancaster and L.~Rodman.
\newblock {\em Algebraic {R}iccati equations}.
\newblock Oxford Science Publications. The Clarendon Press Oxford University
  Press, New York, 1995.

\bibitem{latouchenewton}
G.~Latouche.
\newblock Newton's iteration for non-linear equations in {M}arkov chains.
\newblock {\em IMA J. Numer. Anal.}, 14(4):583--598, 1994.

\bibitem{LuNewton}
L.-Z. Lu.
\newblock Newton iterations for a non-symmetric algebraic {R}iccati equation.
\newblock {\em Numer. Linear Algebra Appl.}, 12(2-3):191--200, 2005.

\bibitem{Lu02}
L.-Z. Lu.
\newblock Solution form and simple iteration of a nonsymmetric algebraic
  {R}iccati equation arising in transport theory.
\newblock {\em SIAM J. Matrix Anal. Appl.}, 26(3):679--685 (electronic), 2005.

\bibitem{OrtegaRheinboldt}
J.~M. Ortega and W.~C. Rheinboldt.
\newblock {\em Iterative solution of nonlinear equations in several variables}.
\newblock Academic Press, New York, 1970.

\bibitem{mythesis}
F.~Poloni.
\newblock {\em Algorithms for quadratic matrix and vector equations}.
\newblock PhD thesis, Scuola Normale Superiore, Pisa, 2010.
\newblock Available at \url{http://fph.altervista.org/acad/index.html}.

\bibitem{thisoldarxiv}
F.~Poloni.
\newblock Quadratic vector equations, 2010.
\newblock arXiv:1004.1500v1. Available at
  \url{http://arxiv.org/abs/1004.1500v1}.

\bibitem{RanReurings}
A.~C.~M. Ran and M.~C.~B. Reurings.
\newblock The symmetric linear matrix equation.
\newblock {\em Electron. J. Linear Algebra}, 9:93--107 (electronic), 2002.

\end{thebibliography}

\end{document}